\documentclass[11pt, a4paper, oneside,leqno]{article}
\usepackage{amsmath, amsfonts, amsthm,dsfont,enumerate, amssymb, fouriernc}
\usepackage{authblk}
\DeclareMathOperator{\sinc}{sinc}

\newtheorem{mylem}{Lemma}[section]
\newtheorem{mythe}{Theorem}[section]

\newtheorem{cor}{Corollary}[section]

\newtheorem{rem}{Remark}[section]

\catcode`,\active

\catcode`\,12

\title{Behaviour of $L_{q}$ norms of the $\sinc_{p}$ function}

\author[1]{David E Edmunds}
\author[2]{Houry Melkonian}
\affil[1]{{\small \it Department of Mathematics, University of Sussex, Brighton BN1 9QH, UK.}}
\affil[2]{{\small \it Department of Mathematics, School of Mathematical and Computer Sciences, Heriot-Watt University, Edinburgh EH14 4AS, UK.}}

\date{April 2018}

\begin{document}
\maketitle
\footnotetext[1]{Email address: \texttt{davideedmunds@aol.com}}
\footnotetext[2]{Email address: \texttt{hm189@hw.ac.uk}}

\begin{abstract}
An integral inequality due to Ball involves the $L_{q}$ norm of
the $\sinc_p$ function; the dependence of this norm on $q$ as $q\rightarrow
\infty$ is now understood. By use of recent inequalities involving $p-$%
trigonometric functions $(1<p<\infty )$ we obtain asymptotic information
about the analogue of Ball's integral when $\sin $ is replaced by $\sin
_{p}.$
\end{abstract}

\section{Introduction}
In \cite{ball1986cube}, K. Ball proved that every section of the unit cube in $\mathbb{R}^n$ by an $(n-1)-$dimensional subspace has $(n-1)-$volume at most $\sqrt{2}$, which is attained if and only if this section contains an $(n-2)-$dimensional face of the cube. To show this, Ball made essential use of the inequality
\[
\sqrt{q}\int_{-\infty}^\infty \left\vert\frac{\sin x}{x}\right\vert^q \mathrm{d}x \leq \sqrt{2} \ \pi, \qquad q \geq 2,
\]
in which equality holds if and only if $q=2$.

It is now known (see \cite{borwein2010p}) that
\begin{align}\label{ballsineq}
\displaystyle\lim_{q \rightarrow \infty}\sqrt{q}\int_{-\infty}^\infty \left\vert\frac{\sin x}{x}\right\vert^q \mathrm{d}x = \sqrt{\frac{3\pi}{2}}.
\end{align}
Moreover, the asymptotic properties of the $q$-norm of the sinc function were studied. In fact, more precise results of an asymptotic nature of the integral in Ball's integral inequality are now known (see \cite{kerman2015asymptotically} for more details).

Stimulated by applications to such differential operators as the p-Laplacian, there is now a large amount of recent work concerning generalisations of the sine function (and other trigonometric functions): see, for example, \cite{EdmundsLang2011}. This encourages us to look for an extension of \eqref{ballsineq} to a more general setting. With this in mind, for $q \in (1, \infty)$, define
\[
I_p(q):= q^{1/p}\int_0^\infty \left\vert \frac{\sin_p x}{x} \right\vert^q \mathrm{d}x
\]
where $p \in (1, \infty)$ and $\sin_p$ is the generalised sine function.  This function is defined to be the inverse of the function $F_p:[0, 1] \longrightarrow [0,\frac{\pi_p}{2}]$ given by
\begin{equation} \label{inversepsine}
 F_p(y):=\int_0^y (1-t^p)^{-\frac{1}{p}}\mathrm{d}t,
\end{equation}
where $\pi_p:=2F_p(1)=\frac{2 \pi}{p \sin(\frac{\pi}{p})}$; it is increasing on $[0,\frac{\pi_p}{2}]$ and is extended to the whole of $\mathbb{R}$ to be a $2\pi_p-$periodic function (still denoted by $\sin_p$) by means of the rules
\begin{equation} \label{sineodd}
    \sin_p(-x)=-\sin_p(x) \qquad \text{and} \qquad \sin_p\left(\frac{\pi_p}{2}-x\right)=\sin_p\left(\frac{\pi_p}{2}+x\right).
\end{equation}
The choice $p=2$ corresponds to the standard trigonometric setting: $\sin_2\equiv \sin$, $\pi_2=\pi$. Moreover, $\pi_p$ is a decreasing function in $p \in (1,\infty)$ such that
\begin{align}\label{piplimits}
\begin{cases}
\pi_p \to \infty& \qquad \text{when} \qquad p \to 1^+\\
\pi_p\to2 &\qquad \text{when} \qquad p\to \infty.
\end{cases}
\end{align}

The main purpose of this paper is to show that for each $%
p\in (1,\infty )$ there is an analogue of \eqref{ballsineq}, namely%
\[
\lim_{q\rightarrow \infty }q^{1/p}\int_{0}^{\infty }\left\vert \frac{\sin
_{p}x}{x}\right\vert ^{q}dx=p^{-1}\left( p\left( p+1\right) \right)
^{1/p}\Gamma \left( 1/p\right) .
\]%
This is achieved by appropriate use of certain recently-obtained
inequalities concerning $\sin _{p}.$ Moreover, information is obtained about
the asymptotic behaviour of the above integral as $q\rightarrow \infty ;$
this complements that known when $p=2.$

\section{Properties of the $\sinc_p$ function}\label{sec2}
Given $p\in (1,\infty ),$ the function sinc$_{p}$ is defined by 
\[
\text{sinc}_{p}x=\left\{ 
\begin{array}{cc}
\frac{\sin _{p}x}{x}, & x\in \mathbb{R}\backslash \{0\}, \\ 
1, & x=0.%
\end{array}%
\right. 
\]%
It is even and its roots are the points $n\pi _{p}$ with $n\in \mathbb{Z}%
\backslash \{0\}.$ Since $\left\vert \sin _{p}x\right\vert \leq 1$ for all $%
x\in \mathbb{R},$ $\lim_{\left\vert x\right\vert \rightarrow \infty }$sinc$%
_{p}x=0.$

\begin{mylem}\label{sincpmonotonic}
\begin{enumerate}[(a)]
\item \label{1a} $\vert \sinc_p x \vert\leq1$ for all $x \in \mathbb{R}$.
\item \label{1b} The function $\sinc_p $ is strictly decreasing on the interval $\left(0, \frac{\pi_p}{2}\right)$.
\end{enumerate}
\end{mylem}
\begin{proof}
{\itshape{(\ref{1a})}} Observe that for $x \in \left(0, \frac{\pi_p}{2}\right]$ we have the $p$-analogue of the classical Jordan inequality \cite[Proposition 2.3]{BushellEdmunds2012},
\begin{align}\label{pjordan}
\frac{2}{\pi_p} \leq \sinc_p x < 1 \qquad \forall x \in\left(0, \frac{\pi_p}{2}\right].
\end{align}
On the other hand, for $x \in \left(\frac{\pi_p}{2}, \infty\right)$ and since $\pi_p \in (2, \infty)$ for all $p \in (1, \infty)$ we conclude that
\[
\vert \sinc_p x\vert=\left\vert \frac{\sin_p x}{x} \right\vert\leq \frac{1}{\vert x \vert}<1.
\]
Since $\sinc_p$ is an even function, {\itshape{(\ref{1a})}} is complete.
\\
{\itshape{(\ref{1b})}} Observe that
\[
\frac{\mathrm{d}}{\mathrm{d}x}\sinc_p x=\frac{\cos_p x}{x^2}(x-\tan_p x)
\]
and $\cos_p x>0$ for any $x \in (0, \frac{\pi_p}{2})$. Let $g(x)=x-\tan_p x$. Then, $g'(x)=-\tan_p^p x<0$ for all $x \in (0, \frac{\pi_p}{2})$ and we conclude that $g(x)$ is strictly decreasing in this interval. Then $g(x)<g(0)=0$ for $x \in (0, \frac{\pi_p}{2})$. The result follows.
\end{proof}
\section{The $p$-version of Ball's integral inequality}\label{sec3}
For $p, q \in (1, \infty)$ define
\[
I_p(q):= q^{1/p}\int_0^\infty \left\vert\frac{\sin_p x}{x}\right\vert^q \mathrm{d}x.
\]
Here we establish the existence of $\displaystyle\lim_{q \rightarrow \infty} I_p(q)$. We use Laplace's method for integrals which suggests approximating the integrand in a neighbourhood by simpler functions for which the integral can be evaluated after proving that the integral on the complementing interval is very small when $q$ is large enough. 
\begin{mylem}\label{alphavanish}
Let $p, q \in (1, \infty)$. For any real $\alpha>0$, we have
\[
\int_\alpha^\infty \left\vert \frac{\sin_p x}{x}\right\vert^q \mathrm{d}x \leq  \begin{cases}
\frac{1}{q-1}\alpha^{1-q}& \quad \text{for} \qquad \alpha \geq 1,\\\\
\left(\frac{\sin_p \alpha}{\alpha}\right)^q(1-\alpha)+\frac{1}{q-1}& \quad \text{for} \qquad \alpha<1.
\end{cases}
\] 
Moreover,
\[
\displaystyle\lim_{q \rightarrow \infty} q^{1/p} \int_\alpha ^\infty \left\vert \frac{\sin_p x}{x} \right\vert^q \mathrm{d}x=0, \qquad \forall \alpha>0.
\]
\end{mylem}
\begin{proof}
We first discuss the case $\alpha \in [1, \infty)$.
\begin{align*}
 \int_\alpha ^\infty \left\vert \frac{\sin_p x}{x} \right\vert^q \mathrm{d}x &= \displaystyle\lim_{\beta \rightarrow \infty}  \int_\alpha ^\beta \left\vert \frac{\sin_p x}{x} \right\vert^q \mathrm{d}x\\
 &\leq \displaystyle\lim_{\beta \rightarrow \infty} \int_\alpha ^\beta x^{-q} \mathrm{d}x\\
 &=\displaystyle\lim_{\beta \rightarrow \infty} \frac{1}{1-q}\left[\frac{1}{\beta^{q-1}}-\frac{1}{\alpha^{q-1}}\right]=\frac{1}{q-1}\alpha^{1-q}.
\end{align*}
Then,
\[
\displaystyle\lim_{q \rightarrow \infty} q^{1/p} \int_\alpha ^\infty \left\vert \frac{\sin_p x}{x} \right\vert^q \mathrm{d}x \leq \displaystyle\lim_{q \rightarrow \infty} \frac{q^{1/p}}{(q-1)\alpha^{q-1}}=0.
\]

For $\alpha \in (0,1)$, we have
\begin{align*}
 \int_\alpha ^\infty \left\vert \frac{\sin_p x}{x} \right\vert^q \mathrm{d}x &= \int_\alpha ^1 \left\vert \frac{\sin_p x}{x} \right\vert^q \mathrm{d}x + \int_1 ^\infty \left\vert \frac{\sin_p x}{x} \right\vert^q \mathrm{d}x.
\end{align*}
From the previous case we have,
\[
q^{1/p} \int_1^\infty \left\vert \frac{\sin_p x}{x} \right\vert^q \mathrm{d}x \leq \frac{q^{1/p}}{q-1},
\]
which approaches zero as $q \rightarrow \infty$.\\
For the remaining integral, the corresponding interval $(\alpha, 1)$ is a subset of $(0, \frac{\pi_p}{2})$ since $\pi_p \in (2, \infty)$. According to Lemma~\ref{sincpmonotonic}, 
\[
0<\frac{\sin_p x}{x}<\frac{\sin_p \alpha}{\alpha}<1, \qquad \forall x \in (\alpha, 1).
\]
Then,
\begin{align*}
  \int_\alpha ^\infty \left\vert \frac{\sin_p x}{x} \right\vert^q \mathrm{d}x&< \left(\frac{\sin_p \alpha}{\alpha}\right)^q(1-\alpha)+\frac{1}{q-1}.
\end{align*}
Using Lemma \ref{sincpmonotonic} {\itshape{(\ref{1a})}}, we conclude that
\[
\displaystyle\lim_{q \rightarrow \infty} q^{1/p} \int_\alpha ^\infty \left\vert \frac{\sin_p x}{x} \right\vert^q \mathrm{d}x  \leq \displaystyle\lim_{k \rightarrow \infty} q^{1/p}\left[\left(\frac{\sin_p \alpha}{\alpha}\right)^q(1-\alpha)+\frac{1}{q-1}\right]=0.
\]
\end{proof}
 Our main result is the following theorem:
\begin{mythe}\label{pBall}
Let $p, q \in (1,\infty)$. Then
\[
\displaystyle\lim_{q \rightarrow \infty} I_p(q)=\frac{1}{p}\Gamma\left(\frac{1}{p}\right)(p(p+1))^{1/p}.
\]
\end{mythe}
\begin{proof}
From \cite[3.251, p. 324]{Grad2007}, for $\mu>0$, $\nu>0$, $\lambda>0$ we have
\begin{align}\label{equ1}
\int_0^1 x^{\mu-1} (1-x^\lambda)^{\nu-1}\mathrm{d}x=\frac{1}{\lambda} B\left(\frac{\mu}{\lambda}, \nu\right).
\end{align}
From \cite[Theorem 1.1 (1) ]{bhayo2012generalized}, for $0<x<\left(1-\left(\frac{2}{\pi_p}\right)^{p(p+1)}\right)^{1/p}$ we have
\begin{align}\label{equ22}
\frac{x}{\sin_p^{-1} x}>(1-x^p)^{\frac{1}{p(p+1)}}.
\end{align}
Note that for all $p \in (1, \infty)$ we have
\[
0<\left(1-\left(\frac{2}{\pi_p}\right)^{p(p+1)}\right)^{1/p}<1.
\]
Let 
\[
\alpha_1:=\sin_p^{-1}\left(1-\left(\frac{2}{\pi_p}\right)^{p(p+1)}\right)^{1/p} \in \left(0, \frac{\pi_p}{2}\right). 
\]
Changing the variable to $y=\sin_p x$ and using the inequality in \eqref{equ22} we see that
\begin{align*}
q^{1/p}\int_0^\infty \left\vert\frac{\sin_p x}{x}\right\vert^q \mathrm{d}x&>q^{1/p}\int_0^{\alpha_1} \left(\frac{\sin_p x}{x}\right)^q \mathrm{d}x\\
&=q^{1/p}\int_0^{\sin_p \alpha_1}\left(\frac{y}{\sin_p^{-1} y}\right)^q (1-y^p)^{-1/p} \mathrm{d}y\\
&>q^{1/p}\int_0^{\sin_p \alpha_1}(1-y^p)^{\frac{q}{p(p+1)}} (1-y^p)^{-1/p} \mathrm{d}y\\
&=q^{1/p}\left(\int_0^1-\int_{\sin_p \alpha_1}^1 (1-y^p)^{\frac{q}{p(p+1)}} (1-y^p)^{-1/p} \mathrm{d}y\right)\\
&=:J_1(p,q)-J_2(p,q).
\end{align*}
From \eqref{equ1} with $\mu=1,\ \lambda=p$ and $\nu=\frac{q}{p(p+1)}-\frac{1}{p}+1$ we get
\begin{align}\label{j1}
\displaystyle\lim_{q \rightarrow \infty}J_1(p, q)&=\displaystyle\lim_{q \rightarrow \infty}q^{1/p}\frac{1}{p} \Gamma\left(\frac{1}{p}\right)\frac{\Gamma\left(\frac{q}{p(p+1)}-\frac{1}{p}+1\right)}{\Gamma\left(\frac{q}{p(p+1)}+1\right)} \notag\\
&=\frac{1}{p} \Gamma\left(\frac{1}{p}\right)\left(p(p+1)\right)^{1/p}.
\end{align}
The last equality is due to the fact that
\begin{align}\label{gamma}
\frac{\Gamma(q+a)}{\Gamma(q+b)} \sim q^{a-b} \quad \text{as}\quad q \rightarrow \infty,
\end{align}
which follows from Stirling's formula: see also \cite[Problem 2, p.45]{henriciapplied}.\\
Moreover,
\begin{align}\label{j2}
\displaystyle\lim_{q \rightarrow \infty}J_2(p, q)&=\displaystyle\lim_{q \rightarrow \infty} q^{1/p}\int_{\sin_p \alpha_1}^1(1-y^p)^{\frac{q}{p(p+1)}} (1-y^p)^{-1/p} \mathrm{d}y\notag\\
&\leq \displaystyle\lim_{q \rightarrow \infty} q^{1/p}(1-\sin_p^p\alpha_1)^{\frac{q}{p(p+1)}}\int_{\sin_p \alpha_1}^1 (1-y^p)^{-1/p} \mathrm{d}y\notag\\
&\leq \displaystyle\lim_{q \rightarrow \infty} q^{1/p}(1-\sin_p^p\alpha_1)^{\frac{q}{p(p+1)}}\int_0^1 (1-y^p)^{-1/p} \mathrm{d}y\notag\\
&=\frac{\pi_p}{2}\displaystyle\lim_{q \rightarrow \infty} q^{1/p}(1-\sin_p^p\alpha_1)^{\frac{q}{p(p+1)}}=0.
\end{align}
Then from \eqref{j1} and \eqref{j2},
\begin{align}\label{inf}
\liminf_{q \rightarrow \infty} q^{1/p}\int_0^\infty \left\vert\frac{\sin_p x}{x}\right\vert^q \mathrm{d}x\geq\frac{1}{p} \Gamma\left(\frac{1}{p}\right)\left(p(p+1)\right)^{1/p}.
\end{align}

On the other hand, from \cite[3.251, p.325]{Grad2007}, for $0 <\mu<p\nu$, $b>0$ and $p>0$ we have
\begin{align}\label{equ2}
\int_0^\infty x^{\mu-1} (1+bx^p)^{-\nu}\mathrm{d}x=\frac{1}{p} b^{-\mu/p} B\left(\frac{\mu}{p}, \nu-\frac{\mu}{p}\right).
\end{align}
From \cite[Theorem 1.1 (1)]{bhayo2012generalized}, for $x \in (0, 1)$ we have
\begin{align}\label{ineqarcsinp}
\frac{x}{\sin_p^{-1} x}< \left(1+\frac{x^p}{p(p+1)}\right)^{-1}.
\end{align}
Now let $\alpha_2 \in (0, \pi_p/2]$. Changing the variable to $y=\sin_p x$ and using the inequality in \eqref{ineqarcsinp} we obtain
\begin{align}\label{equ4}
\int_0^{\alpha_2}\left(\frac{\sin_p x}{x}\right)^q \mathrm{d}x&=\int_0^{\sin_p \alpha_2} \left(\frac{y}{\sin_p^{-1} y}\right)^q (1-y^p)^{-1/p} \mathrm{d}y\notag\\
&< \int_0^{\sin_p \alpha_2}  \left(1+\frac{y^p}{p(p+1)}\right)^{-q} (1-y^p)^{-1/p} \mathrm{d}y\notag\\
&<\left(1-\sin_p^p \alpha_2\right)^{-1/p}  \int_0^{\sin_p \alpha_2}  \left(1+\frac{y^p}{p(p+1)}\right)^{-q} \mathrm{d}y\notag\\
&=:\left(1-\sin_p^p \alpha_2\right)^{-1/p} J_3(p, q).
\end{align}
From \eqref{equ2}, with $\mu=1,\ b=(p(p+1))^{-1}$ and $\nu=q \in (1/p, \infty)$, we obtain
\begin{align}\label{equ5}
J_3(p, q)&<\frac{1}{p} (p(p+1))^{1/p} B\left(\frac{1}{p}, q-\frac{1}{p}\right)\notag \\
&=\frac{1}{p} (p(p+1))^{1/p} \Gamma\left(\frac{1}{p}\right)\frac{\Gamma\left(q-\frac{1}{p}\right)}{\Gamma(q)}.
\end{align}
With the help of \eqref{gamma} it follows that
\[
\limsup_{q \rightarrow \infty} q^{1/p} \int_0^{\alpha_2}\left\vert\frac{\sin_p x}{x}\right\vert^q \mathrm{d}x \leq \left(1-\sin_p^p \alpha_2\right)^{-1/p} \frac{1}{p} \left(p(p+1)\right)^{1/p} \Gamma\left(\frac{1}{p}\right).
\]
Letting $\alpha_2 \rightarrow 0^+$ we conclude that,
\[
\limsup_{q \rightarrow \infty} q^{1/p} \int_0^{\alpha_2}\left(\frac{\sin_p x}{x}\right)^q \mathrm{d}x \leq \frac{1}{p}\left(p (p+1)\right)^{1/p} \Gamma\left(\frac{1}{p}\right).
\]
From Lemma \ref{alphavanish} we conclude that
\begin{align}\label{sup}
\limsup_{q \rightarrow \infty}  q^{1/p}\int_0^\infty\left\vert\frac{\sin_p x}{x}\right\vert^q \mathrm{d}x&=\limsup_{q \rightarrow \infty} q^{1/p} \left(\int_0^{\alpha_2}+\int_{\alpha_2}^\infty\left\vert\frac{\sin_p x}{x}\right\vert^q \mathrm{d}x\right)\notag\\
&\leq\frac{1}{p}\left(p (p+1)\right)^{1/p} \Gamma\left(\frac{1}{p}\right).
\end{align}
From \eqref{inf} and \eqref{sup} we deduce that
\[
\displaystyle\lim_{q \rightarrow \infty} I_p(q)=\frac{1}{p}\Gamma\left(\frac{1}{p}\right)(p(p+1))^{1/p}.
\]
\end{proof}
\section{Asymptotic Expansion}
Here we investigate the asymptotic expansion of the $p-$Ball integral $I_p(q)$ by performing explicit calculations leading to a precise knowledge of the first two coefficients of the expansion. The study involved provides another proof of Theorem \ref{pBall}; the technique used is an adaptation of that developed in \cite{kerman2015asymptotically}.
\begin{mythe}\label{asyexp}
There exist constants $\gamma_3, \gamma_4, \ldots$ such that for $q$ large enough
\begin{eqnarray*}\nonumber
I_p(q) \sim \left(p(p+1)\right)^{1/p}\left( \frac{1}{p}\Gamma\left(\frac{1}{p}\right)+ \frac{1}{p}\Gamma\left(\frac{1}{p}\right)\frac{(-p^2+p+1)(p+1)}{2p(2p+1)}\frac{1}{q}+\sum_{j=3}^{\infty} \Gamma\left(j+1/p\right)\frac{\gamma_j}{q^j}\right).
\end{eqnarray*}
\end{mythe}
\begin{proof}
For $\alpha \in (0, 1)$, let 
\begin{align*}
J(q, \alpha)&:=q^{1/p}\int_0^\alpha \left(\frac{\sin_p x}{x}\right)^q \mathrm{d}x \\
&=q^{1/p}\int_0^\alpha \exp\left(\frac{-q x^p}{p(p+1)} \right)\left[ \exp\left(\frac{x^p}{p(p+1)} \right)\frac{\sin_p x}{x} \right]^q \mathrm{d}x.
\end{align*}
By Lemma \ref{alphavanish},
\begin{align*}
q^{1/p}\int_\alpha^\infty \left\vert\frac{\sin_p x}{x}\right\vert^q \mathrm{d}x\leq q^{1/p}\left(\frac{\sin_p \alpha}{\alpha}\right)^q(1-\alpha)+\frac{q^{1/p}}{q-1}.
\end{align*}
It is therefore enough to establish the existence of constants $\gamma_3, \gamma_4, \ldots$ such that
\begin{eqnarray*}
J(q, \alpha)\sim \left(p(p+1)\right)^{1/p}\left( \frac{1}{p}\Gamma\left(\frac{1}{p}\right)+ \frac{1}{p}\Gamma\left(\frac{1}{p}\right)\frac{(-p^2+p+1)(p+1)}{2p(2p+1)}\frac{1}{q}+\sum_{j=3}^{\infty} \Gamma\left(j+1/p\right)\frac{\gamma_j}{q^j}\right).
\end{eqnarray*}
Changing the variable to $u=\frac{x}{\left(p(p+1)\right)^{1/p}}$ yields
\begin{align*}
J(q, \alpha)&=q^{1/p} \left(p(p+1)\right)^{1/p}\int_0^{\frac{\alpha}{\left(p(p+1)\right)^{1/p}}} \exp\left(-q u^p\right)\left[ \exp\left(u^p\right)\frac{\sin_p \left(\left(p(p+1)\right)^{1/p} u\right)}{\left(p(p+1)\right)^{1/p} u} \right]^q \mathrm{d}u.
\end{align*}
For the exponential term we have 
\[
\exp\left(u^p\right)=\sum_{j=0}^\infty \frac{u^{pj}}{j!}.
\]
While for $\frac{\sin_p \left(\left(p(p+1)\right)^{1/p} u\right)}{\left(p(p+1)\right)^{1/p} u} $, we have from \cite[(2.17)]{BushellEdmunds2012} the power series expansion of $\sin_p^{-1} x$, and by the Lagrange reversion theorem this gives the existence of constants $a_j$ such that 
\[
\frac{\sin_p \left(\left(p(p+1)\right)^{1/p} u\right)}{\left(p(p+1)\right)^{1/p} u} =\sum_{j=0}^\infty a_j \left(p(p+1)\right)^{j} u^{pj};
\]
the series converges for sufficiently small $u$. The coefficients of the first three terms of this expansion involve $a_0=1$, $a_1=\frac{-1}{p(p+1)}$ and $a_2=\frac{-p^2+2p+1}{2p^2(p+1)(2p+1)}$. The Cauchy product formula then gives,
\[
\exp\left(u^p\right)\frac{\sin_p \left(\left(p(p+1)\right)^{1/p} u\right)}{\left(p(p+1)\right)^{1/p} u} =1+\sum_{j=2}^\infty b_j u^{pj},
\]
where
\[
b_j=\sum_{l=0}^j\frac{a_{j-l} \left(p(p+1)\right)^{j-l} }{l!}
\]
and the power series converges for sufficiently small $u$.

We know that for small values of $u$,
\[
\left\vert\exp\left(u^p\right)\frac{\sin_p \left(\left(p(p+1)\right)^{1/p} u\right)}{\left(p(p+1)\right)^{1/p} u}-1\right\vert=\left\vert\sum_{j=2}^\infty b_j u^{pj}\right\vert\leq \sum_{j=2}^\infty \left\vert b_j \right\vert u^{pj}<1.
\]
Note that the power series $\sum_{j=2}^\infty b_j u^{pj}$ is absolutely convergent for sufficiently small $u$. 

Therefore by the Binomial expansion we get
\begin{align*}
\left[\exp\left(u^p\right)\frac{\sin_p \left(\left(p(p+1)\right)^{1/p} u\right)}{\left(p(p+1)\right)^{1/p} u}\right]^q&=1+q\left[\sum_{j=2}^\infty b_j u^{pj}\right]+\frac{q(q-1)}{2}\left[\sum_{j=2}^\infty b_j u^{pj}\right]^2\\
&+\cdots+\frac{q(q-1)\ldots(q-m+1)}{m!}\left[\sum_{j=2}^\infty b_j u^{pj}\right]^m+\cdots
\end{align*}
Since the right hand side of the Binomial expansion is bounded from above by
\begin{align*}
1&+q\left[\sum_{j=2}^\infty \left\vert b_j\right\vert u^{pj}\right]+\frac{q(q-1)}{2}\left[\sum_{j=2}^\infty \left\vert b_j \right\vert u^{pj}\right]^2\\
&+\cdots+\frac{q(q-1)\ldots(q-m+1)}{m!}\left[\sum_{j=2}^\infty \left\vert b_j\right\vert  u^{pj}\right]^m+\cdots=\left[1+\sum_{j=2}^\infty \left\vert b_j \right\vert u^{pj}\right]^q,
\end{align*}
we may rearrange terms and, for small enough $u$, obtain \eqref{expansion} .
Hence,
\begin{align}\label{expansion}
\left[\exp\left(u^p\right)\frac{\sin_p \left(\left(p(p+1)\right)^{1/p} u\right)}{\left(p(p+1)\right)^{1/p} u}\right]^q=\sum_{j=0}^{\infty}  c_j u^{pj},
\end{align}
where $c_0=1$, $c_1=0$ and $c_2=q b_2=\frac{qp(-p^2+p+1)}{2(2p+1)}$. Observe that the other coefficients $c_j=c_j(q) \ (j \geq 3)$ can be obtained by the following rearrangements:
\begin{align*}
qb_3u^{3p}=c_3 u^{3p}, \quad  \left(qb_4+\frac{q(q-1)}{2}b_2^2\right)u^{4p}&=c_4 u^{4p}, \ldots
\end{align*}

Specialised to our case, \cite[Theorem 8.1, p. 86]{olver1974asymptotics} (with $x=q, \ p(t)=u^p, \ q(t)= \exp\left(u^p\right)\frac{\sin_p \left(\left(p(p+1)\right)^{1/p} u\right)}{\left(p(p+1)\right)^{1/p} u}, \ s=pj,\ \lambda=1$ and $\mu=p$) establishes the existence of real constants $\gamma_0, \gamma_1, \ldots$ such that
\begin{align*}
J(q, \alpha) &\sim q^{1/p}\left(p(p+1)\right)^{1/p}\sum_{j=0}^\infty \Gamma\left(j+1/p\right)\frac{\gamma_j}{q^{j+1/p}}\\
&=\left(p(p+1)\right)^{1/p}\sum_{j=0}^\infty \Gamma\left(j+1/p\right)\frac{\gamma_j}{q^{j}}, \quad (q \rightarrow \infty)
\end{align*}
where 
\begin{equation*}
\gamma_0=\frac{1}{p}, \quad \gamma_1=0 \quad \text{and} \quad \gamma_2=\frac{q(-p^2+p+1)}{2(2p+1)}.
\end{equation*}
\end{proof}
\begin{rem}
The asymptotic expansion in Theorem \ref{asyexp} complements that of \cite{kerman2015asymptotically} when $p=2$; it involves the coefficients $b_j$ of the expansion of \linebreak$\exp\left(u^p\right)\frac{\sin_p \left(\left(p(p+1)\right)^{1/p} u\right)}{\left(p(p+1)\right)^{1/p} u}$ which depend on the constants $a_j$ of the power series of the function $\sinc_p$. So far the first three terms in the expansion of $\sin_p$ are known and no regular pattern has been obtained for the other subsequent terms. It remains to see whether or not higher-order terms in the expansion of $I_p(q)$ can be determined.
\end{rem}
\section{Concluding remarks}
In this section we present some results obtained from Theorem \ref{pBall}. The proofs are natural adaptations of those
given in \cite{borwein2010p} and are therefore omitted. 

For $q\in(1, \infty)$ and $n \in \mathbb{N}\cup \{0\}$, define 
\[
\varphi_p(n, q):=\int_0^\infty \left(\ln \left\vert\frac{\sin_p x}{x}\right\vert\right)^n \left\vert \frac{\sin_p x}{x}\right\vert^q\ \mathrm{d}x.
\]
Note that $\varphi_p(0,q):=\varphi_p(q)=I_p(q)$.

A more general result of $p-$Ball integral inequality can also be achieved by induction for any non-negative integer $n$.
\begin{mylem}
For $n \in \mathbb{N}\cup \{0\}$ and $p\in (1, \infty)$. Then
\begin{align*}
\displaystyle\lim_{q \rightarrow \infty} q^{n+\frac1p}\varphi_p(n, q)=(-1)^n\frac{1}{p}\Gamma\left(\frac{1}{p}\right)(p(p+1))^{1/p} \ \Gamma\left(n+\frac1p\right).
\end{align*}
\end{mylem}

The following gives the analyticity of the function $\varphi_p(q)$ in a region containing $(1, \infty)$. The proof makes use of the $L_q-$Lebesgue integrability of the $\sinc_p$ functions when $p, q \in (1, \infty)$.
\begin{cor}
Let $q\in (1, \infty)$. For $1-q<z<q-1$,
\[
\varphi_p(q-z)=\sum_{n=0}^\infty (-1)^n \varphi_p(n, q) \frac{z^n}{n!},
\]
where $\varphi_p^{(n)}(q)=\varphi_p(n, q)$.
\end{cor}


\end{document}